\documentclass[10pt,reqno]{amsart}
\pdfoutput=1
\usepackage{graphicx}
\usepackage{indentfirst,csquotes}
\usepackage[colorlinks = true,
            linkcolor = blue,
            urlcolor  = blue,
            citecolor = blue,
            anchorcolor = blue]{hyperref}
\usepackage{amssymb,amsthm,amsmath}
\usepackage{xcolor,paralist,titlesec,fancyhdr,etoolbox}
\usepackage{algorithm,algpseudocode}
\allowdisplaybreaks

\newtheorem{theorem}{Theorem}[]

\newtheorem{lemma}[theorem]{Lemma}

\newtheorem{corollary}[theorem]{Corollary}

\newcommand{\myceil}[1]{\left \lceil #1 \right \rceil }
\newcommand{\myfloor}[1]{\left \lfloor #1 \right \rfloor }

\titleformat{\section}{\Large\bfseries}
\titlespacing{\section}{\thesection. }

\hypersetup{ colorlinks=true, linkcolor=black, filecolor=black, urlcolor=black }

\usepackage{lipsum}

\begin{document}

\author{Yann Dijoux}
\address{Yann Dijoux, LIST3N Research Unit, Université de Technologie de Troyes, France}
\email[Y. Dijoux]{yann.dijoux@utt.fr}
\date{\today}
\title{Chebyshev polynomials involved in the Householder's
method for square roots}

\let\thefootnote\relax

\begin{abstract}
  The Householder's method is a root-find algorithm which is a natural extension of the methods of Newton and Halley. 	The current paper mostly focuses on approximating the square root of a positive real number based on these methods. The resulting algorithms can be expressed using Chebyshev polynomials. An extension to the nth root is also proposed.	 
\end{abstract} 

\maketitle

\section{Motivation}

The Babylonian method \cite{Kosheleva2009}, also called Heron's method \cite{Fowler1998366}, is a particular case of the Newton's method \cite{NewtonPol} in order to compute the square root of a positive real number $x$. The method consists of defining  a sequence $\{u_n\}_{n \geq 0}$ converging quadratically to $\sqrt{x}$. The sequence is initialized by $u_0=r$, where $r$ is commonly set to either 1 or to the nearest integer from $\sqrt{x}$, the latter option being opted in the following examples. The remainder of the sequence is defined iteratively as:

\begin{equation}
(\forall n \geq 0) \;\;\; u_{n+1} = \frac{u_n}{2}+\frac{x}{2u_n}
\label{FirstNewt}
\end{equation}

For example, for $x=51$, $r$ can be set to to 7 and the sequence starts by $7,\frac{50}{7}$, $\frac{4999}{7000}$,$\frac{49980001}{6998600}$. In the following, the values $b=\frac{x-r^2}{2}$ and $M=\frac{x+r^2}{2}$ are introduced. In the previous example, $b=1$ and $M=50$. $b$ and $M$ correspond respectively to half the difference and the middle value between $x$ and $r^2$. The motivation of this work was to provide an explicit expression of $u_n$, ideally of the form $u_n=u_{n-1}+\varepsilon_n$, where $\varepsilon_n$ depends on $x$ and $r$, or additionally on $M$ and $b$.

The second term of the sequence $u_1$ can be expressed as $u_1=r+\frac{b}{r}$ or equivalently to $u_1=\frac{M}{r}$. The third term of the sequence $u_2$ can be expressed as followed:

\begin{equation}
u_2= r+\frac{b}{r} - \frac{b^2}{2rM}
\end{equation}

The expression of $u_2$ is equivalent to the Bakhshali method \cite{Gadtia20201,Shirali2012884}. This method can be  obtained either by iterating the Newton's method twice or by the second-order Taylor expansion of $\sqrt{x}=\frac{M}{r}\sqrt{1-\left(\frac{b}{M}\right)^2}$. Considering our numerical example, we obtain $u_2=7+\frac{1}{7}-\frac{1}{700}$. Note that the second-order Taylor expansion  of $\sqrt{x}=\sqrt{r^2+2b}$ is less effective than the Bakhshali method.

After a straightforward computation, the fourth term of the sequence $u_3$ can be expressed as follows:

\begin{equation}
u_3= r+\frac{b}{r} - \frac{b^2}{2rM} - \frac{b^4}{4rM(2M^2-b^2)}
\end{equation}

In the numerical example, we obtain $u_3=7+\frac{1}{7}-\frac{1}{700}-\frac{1}{6998600}$. As $b$ is commonly negligible compared to $M$ ($M \gg b$), the expression of $u_3$ can be approximated by $\widehat{u_3}$ as follows:

\begin{equation}
\widehat{u_3}= r+\frac{b}{r} - \frac{b^2}{2rM} - \frac{b^4}{8rM^3}
\end{equation}

This expression is a bit more tractable, and the numerical example is simpler here with $\widehat{u_3}=7+\frac{1}{7}-\frac{1}{700}-\frac{1}{7000000}$. The expression of $\widehat{u_3}$ can be obtained from the Taylor expansion of $\sqrt{x}=\frac{M}{r}\sqrt{1-\left(\frac{b}{M}\right)^2}$.  In particular, this Taylor expansion can be extended further as follows:

\begin{equation}
\widehat{u_n}= r+\frac{b}{r} - \frac{b^2}{2rM} - \frac{b^4}{8rM^3}- \frac{b^6}{16rM^5}-\frac{5}{128}\frac{b^8}{rM^7}-\ldots -\frac{\binom{2n}{n}}{4^n(2n-1)}\frac{b^{2n}}{rM^{2n-1}}
\end{equation}

However, the convergence of this Taylor expansion is much slower than the sequence from the Newton's method as it involves all even powers of $b$ while the Newton's method involves only power of two powers of $b$. The expression of $u_4$ and $u_5$ are presented in (\ref{labu4}) and (\ref{labu5}), respectively.

  \begin{equation}
u_4= r+\frac{b}{r} - \frac{b^2}{2rM} - \frac{b^4}{4rM(2M^2-b^2)}-\frac{b^8}{8rM(2M^2-b^2)(8M^4-8M^2b^2+b^4)}
\label{labu4}
\end{equation}
	
  \begin{equation}
u_5= u_4 -\frac{b^{16}}{16rM(2M^2-b^2)(8M^4-8M^2b^2+b^4)(128M^8-256M^6b^4+160M^4b^4-32M^2b^6+b^8)}
\label{labu5}
\end{equation}
	
The identification of the first Chebyshev polynomials of the first kind at the denominators of the sequence has initiated the present work. The current paper is organized as follows. Section 2 presents the expression of the Newton's method for square roots based on Chebyshev polynomials of the first and second kinds. The Halley's method for square roots is presented in Section 3 and involves the Chebyshev polynomials of the third and fourth kind. In Section 4, the Householder's method for square roots is proposed and various expressions of the algorithm are derived. An extension of the Householder's method for nth root is introduced in Section 5. Finally, a note on the asymptotic expansion of Chebyshev functions is presented in Section 6.

\section{Polynomial expressions in the Babylonian method}

The first polynomials at stake at the denominators of $u_n$ are respectively $T_1(X)=X$, $T_2(X)=2X^2-1$, $T_4(X)=8x^4-8x^2+1$ and $T_8(X)=128x^8-256x^6+160x^4-32x^2+1$ which are Chebyshev polynomials of the first kind \cite{Mason20021}.

The Chebyshev polynomials of the first kind are defined by $T_n(\textit{cos}t)=\textit{cos}(nt)$. They can be defined by the following recurrence relation:

$$T_0(X)=1\;\;,\;\; T_1(X)=X \;\;,\;\; T_{n+1}(X)=2XT_n(X)-T_{n-1}(X)$$

A useful relationship can be derived from the well-known  identity: $2T_n(X)T_m(X)=T_{m+n}(X)+T_{ |m-n|}(X)$ and is presented as follows:

  \begin{equation}
T_{2^{n+1}}(X)=2T^2_{2^n}(X)-1
\label{powertn}
\end{equation}

From the expression of $u_5$ in  (\ref{labu5}), we can conjecture an expression of $u_n$ based on the Chebyshev polynomials of index of power of two. Theorem \ref{th:conjTche} presents the explicit expression of the sequence $(u_n)$ and a proof is presented, after deriving two basic properties of the 
Chebyshev polynomials in Lemma  \ref{lm:conjTche} and  Lemma \ref{lm:conjTche2}.

\begin{theorem}\label{th:conjTche}
 Let $x$ be a real positive number  and $\{u_n\}_{n\geq 0}$ be the sequence associated to the Babylonian method reminded in (\ref{babth}), where $r$ is a real positive number.

\begin{equation}
u_0=r \;\;\;,\;\;(\forall n \geq 0)\;\;u_{n+1}= \frac{u_n}{2}+\frac{x}{2u_n}
\label{babth}
\end{equation}

An explicit expression of $u_n$ is presented as follows for $n$ greater or equal than 1, based on the Chebyshev polynomials of the first kind $\{T_k\}_{k \geq 0}$:

\begin{equation}
u_n=\frac{x+r^2}{2r}-\frac{x-r^2}{2r}\displaystyle \sum_{k=1}^{n-1} \displaystyle \frac{1}{2^k \displaystyle \prod_{j=0}^{k-1} T_{2^j}(\frac{x+r^2}{x-r^2})}
\label{exprexpl}
\end{equation}

Furthermore,  if we denote $b=\frac{x-r^2}{2}$ and $M=\frac{x+r^2}{2}$, the previous expression of $u_n$ can be derived as follows:

\begin{equation}
u_n=\frac{M}{r}-\frac{b}{r}\displaystyle \sum_{k=1}^{n-1} \displaystyle \frac{1}{2^k \displaystyle \prod_{j=0}^{k-1} T_{2^j}\left(\frac{M}{b}\right)}
\label{exprexpl2}
\end{equation}

Finally, an expression of $u_n$ over a common denominator is expressed for $n$ greater or equal than 1:

\begin{equation}
u_n=\frac{b}{r}\displaystyle  \frac{T_{2^{n-1}}\left(\frac{M}{b}\right)}{2^{n-1} \displaystyle \prod_{j=0}^{n-2} T_{2^j}\left(\frac{M}{b}\right)}
\label{denom}
\end{equation}

\end{theorem}

\begin{lemma}\label{lm:conjTche}
The Chebyshev polynomials verify this equality for all $N$ greater or equal than 1:
\begin{equation}
T_{2^N}(X)=2^N X \prod_{j=0}^{N-1}T_{2^j}(X)-\sum_{j=1}^N 2^{N-j}\prod_{k=j}^{N-1}T_{2^k}(X)
\label{lemma1eq}
\end{equation}
\end{lemma}

\begin{lemma}\label{lm:conjTche2}
 The Chebyshev polynomials verify this equality for all $N$ greater or equal than 0:
\begin{equation}
T_{2^N}^2(X)-4^N( X^2-1) \prod_{i=0}^{N-1}T_{2^i}^2(X)=1
\label{lemma2eq}
\end{equation}
\end{lemma}

\begin{proof} \label{pr:conjTche}
  The proofs of Lemma \ref{lm:conjTche} and Lemma \ref{lm:conjTche2} are trivial. Both base cases are immediate. For each lemma, the induction step requires to combine the induction hypothesis with the equality highlighted in (\ref{powertn}).
	
	Next, we will prove that the expressions in (\ref{exprexpl2}) and (\ref{denom}) are identical. \\ Let $v_n$ be $\frac{M}{r}-\frac{b}{r}\displaystyle \sum_{k=1}^{n-1} \displaystyle \frac{1}{2^k \displaystyle \prod_{j=0}^{k-1} T_{2^j}\left(\frac{M}{b}\right)}$. $v_n$ can be rewritten as follows:

$$v_n=\frac{2^{n-1}M\displaystyle \prod_{j=0}^{n-2} T_{2^j}\left(\frac{M}{b}\right)-b\sum_{j=1}^{n-1}2^{n-1-j}\prod_{k=j}^{n-2} T_{2^k}\left(\frac{M}{b}\right)}{2^{n-1} r \displaystyle \prod_{j=0}^{n-2} T_{2^j}\left(\frac{M}{b}\right)}$$
$$v_n=\frac{b \left(2^{n-1}\frac{M}{b}\displaystyle \prod_{j=0}^{n-2} T_{2^j}\left(\frac{M}{b}\right)-\sum_{j=1}^{n-1}2^{n-1-j}\prod_{k=j}^{n-2} T_{2^k}\left(\frac{M}{b}\right)\right)}{2^{n-1} r \displaystyle \prod_{j=0}^{n-2} T_{2^j}\left(\frac{M}{b}\right)}$$	
	\begin{equation}
v_n=\frac{b}{r}\displaystyle  \frac{T_{2^{n-1}}\left(\frac{M}{b}\right)}{2^{n-1} \displaystyle \prod_{j=0}^{n-2} T_{2^j}\left(\frac{M}{b}\right)}
\label{denompr}
\end{equation}
	
The last equality in (\ref{denompr}) has been obtained using Lemma \ref{lm:conjTche} with $X=\frac{M}{b}$ and $N=n-1$. The next step consists of evaluating the quantity $\frac{v_n^2-M-b}{2v_n}$ using the last version of $v_n$ from (\ref{denompr}):

$$\frac{ \displaystyle v_n^2-M-b}{2v_n}=\frac{ \displaystyle b^2T^2_{2^{n-1}}\left(\frac{M}{b}\right)-r^2(M+b)4^{n-1}\prod_{j=0}^{n-2} T^2_{2^j}\left(\frac{M}{b}\right)}{ \displaystyle 2 r b 2^{n-1}T_{2^{n-1}}\left(\frac{M}{b}\right)\prod_{j=0}^{n-2} T_{2^j}\left(\frac{M}{b}\right)}$$
$$=\frac{\displaystyle b^2T^2_{2^{n-1}}\left(\frac{M}{b}\right)-4^{n-1}(M^2-b^2)\prod_{j=0}^{n-2} T^2_{2^j}\left(\frac{M}{b}\right)}{\displaystyle  r b 2^{n}\prod_{j=0}^{n-1} T_{2^j}\left(\frac{M}{b}\right)}$$
$$=\frac{b\left(T^2_{2^{n-1}}\left(\frac{M}{b}\right)-4^{n-1}(\left(\frac{M}{b}\right)^2-1)\prod_{j=0}^{n-2} T^2_{2^j}\left(\frac{M}{b}\right)\right)}{\displaystyle    2^{n} r \prod_{j=0}^{n-1} T_{2^j}\left(\frac{M}{b}\right)}$$
\begin{equation}
\frac{v_n^2-M-b}{2v_n}=\frac{b}{\displaystyle    2^{n} r \prod_{j=0}^{n-1} T_{2^j}\left(\frac{M}{b}\right)}
\label{prmil}
\end{equation}
	
The last equality in (\ref{prmil}) has been obtained using Lemma \ref{lm:conjTche2} with $X=\frac{M}{b}$ and $N=n-1$.	Finally, we prove by induction that $u_n=v_n$ for $n$ greater or equal than 1. The base case is immediate as both quantities are equal to $\frac{M}{r}$. The induction hypothesis consists of assuming that $u_n=v_n$ for  a positive integer $n$. The induction step can be derived as follows:

$$u_{n+1}=\frac{u_n}{2}+\frac{x}{2u_n}=u_n-\frac{u_n^2-M-b}{2u_n}=v_n-\frac{v_n^2-M-b}{2v_n}$$
\begin{equation}
=\frac{M}{r}-\frac{b}{r}\displaystyle \sum_{k=1}^{n-1} \displaystyle \frac{1}{2^k \displaystyle \prod_{j=0}^{k-1} T_{2^j}\left(\frac{M}{b}\right)}-\frac{b}{\displaystyle    2^{n} r \prod_{j=0}^{n-1} T_{2^j}\left(\frac{M}{b}\right)}=\frac{M}{r}-\frac{b}{r}\displaystyle \sum_{k=1}^{n} \displaystyle \frac{1}{2^k \displaystyle \prod_{j=0}^{k-1} T_{2^j}\left(\frac{M}{b}\right)}=v_{n+1}
\label{prfin}
\end{equation}
This ends the proof of Theorem \ref{th:conjTche}.

\end{proof}

An algorithm to compute $u_n$ can be derived for $n$ greater or equal than 1. It is implemented in Algorithm 1 with input $x$, an initialization $r$, a positive index $n$ and with output $S$. A similar algorithm involving the computation of square root,   Chebyshev polynomials and the ratio $\frac{x+r^2}{x-r^2}$ has been obtained through an Engel expansion \cite{PerronEngelp127}:
$$\sqrt{\frac{x+r^2}{x-r^2}}=\displaystyle \prod_{k=0}^\infty \left(1+\frac{1}{X_k}\right) \;\;\;,\; X_0=\frac{x}{r^2} \;\;\;,\; X_{k+1}=T_2(X_k)$$


\begin{algorithm}[!ht]
\label{algoBab}
\caption{$\;\;\;S=u(x,r,n)$}
\begin{algorithmic}[1]
\item $S \gets \frac{x+r^2}{2r}$
\item $X \gets  \frac{x+r^2}{x-r^2}$
\item $e \gets \frac{r^2-x}{4rX}$
\For{i=2:n}
\item $S \gets S+e$
\item $X \gets 2X^2-1$
\item $e \gets \frac{e}{2X}$      
\EndFor
\State ${\bf return} \displaystyle \;\; S$
\end{algorithmic}
\end{algorithm}

The previous expressions are only expressed using the Chebyshev polynomials of the first kind  but can also be expressed using the Chebyshev polynomials of the second kind \cite{Mason20021} $U_n$, which is a family of polynomials defined as follows:

\begin{equation}
U_0(X)=1\;\;,\;\; U_1(X)=2X \;\;,\;\; U_{n+1}(X)=2XU_n(X)-U_{n-1}(X)
\label{defUn}
\end{equation}

A consequence of Theorem \ref{th:conjTche} is presented in Corollary \ref{th:conjTcheCoro}, followed by a proof.

\begin{corollary}\label{th:conjTcheCoro} The sequence $\{u_n\}_{n\geq 0}$ can be expressed as follows:

\begin{equation}
u_0=r \;\;,\;\; (\forall \, n \in \mathbb{N}^* \,) \;\; u_n =r \left. \frac{T_{2^{n-1}}\left(X\right)}{(X-1)U_{2^{n-1}-1}\left(X\right)}\right|_{X=\left(\frac{x+r^2}{x-r^2}\right)}
\label{newtU}
\end{equation}

Alternately,  the sequence $\{u_n\}_{n\geq 0}$ can be expressed as a product of monomials:

\begin{equation} 
u_n= \displaystyle r  \left. \displaystyle \prod_{k=0}^{2^{n-1}-1} \frac{\displaystyle X-cos\left( \frac{(2 k+1) \pi}{2^n} \right)}{\displaystyle X-cos\left( \frac{2 k \pi}{2^n} \right)} \right|_{X=\left(\frac{x+r^2}{x-r^2}\right)}
\label{NewtMON}
\end{equation}

\end{corollary}

\begin{proof} \label{pr2}

 A classical identity of the Chebyshev polynomials \cite{Rayes20051231} is that $U_{2n-1}=2T_n U_{n-1}$. It implies that $U_{2^{n}-1}= 2^n \prod_{j=0}^{n-1} T_{2^{j}}$. Therefore, Equation (\ref{denom}) can be re-written as follows:

\begin{equation}
u_n=\frac{b}{r}\displaystyle  \frac{T_{2^{n-1}}\left(\frac{M}{b}\right)}{U_{2^{n-1}-1}\left(\frac{M}{b}\right)}
\label{denom22}
\end{equation}

As $r^2=M-b$, the previous equality can be expressed as follows:

\begin{equation}
u_n=r\displaystyle  \frac{T_{2^{n-1}}\left(\frac{M}{b}\right)}{\left(\frac{M}{b}-1\right)U_{2^{n-1}-1}\left(\frac{M}{b}\right)}
\label{denom33}
\end{equation}

This proves the expression (\ref{newtU}) by replacing $M$ and $b$ in function of $x$ and $r$. The roots of Chebyshev polynomials $T_j$ and $U_j$ are $\left\{cos(\frac{(2k+1)\pi}{2j})\right\}_{k=0..j-1}$ and $\left\{cos(\frac{k\pi}{j+1})\right\}_{k=1..j}$, respectively \cite{Mason20021}. Therefore the roots of the rational function $\displaystyle \frac{T_{2^{n-1}}\left(X\right)}{(X-1)U_{2^{n-1}-1}\left(X\right)}$ are $\left\{cos\left( \frac{(2 k+1) \pi}{2^n} \right)\right\}_{k=0..2^{n-1}-1}$ and their poles are 1 and $\left\{cos\left( \frac{(2 k) \pi}{2^n} \right)\right\}_{k=1..2^{n-1}-1}$.Writing 1 as 
$cos\left( \frac{(2\times 0) \pi}{2^n} \right)$ and evaluating the rational function at ${X=\left(\frac{x+r^2}{x-r^2}\right)}$ allows to obtain (\ref{NewtMON}).

\end{proof}

\section{Polynomial expressions in the Halley's method for square roots}

 While the Newton's method provides a quadratic rate of convergence, the Halley's method \cite{Hasan20066379} has a cubic rate of convergence to the root. Its corresponding algorithm for computing the square root of $x$ with an initial value $r$ is characterized through a sequence $\{{h_n}\}_{n \geq 0}$ defined as follows:

\begin{equation}
{h_0}=r \;\;\;,\;\; (\forall n \geq 0) \;\;\; {h_{n+1}}= {h_{n}} \frac{{h_{n}}^2+3x}{3{h_{n}}^2+x}
\label{defHall}
\end{equation}

As previously, the sequence will be initially written in function of $b=\frac{x-r^2}{2}$ and $M=\frac{x+r^2}{2}$. The first terms of the sequences can be expressed as follows:

$$h_1=r\frac{2M+b}{2M-b}$$
$$h_2=r\frac{(2M+b)}{(2M-b)}\frac{(8M^3-6Mb^2+b^3)}{(8M^3-6Mb^2-b^3)}$$
\begin{equation}
{h_3}=r\frac{(2M+b)}{(2M-b)}\frac{(8M^3-6Mb^2+b^3)}{(8M^3-6Mb^2-b^3)}\frac{(512M^9-1152M^7b^2+864M^5b^4-240M^3b^6+18Mb^8+b^9)}{(512M^9-1152M^7b^2+864M^5b^4-240M^3b^6+18Mb^8-b^9)}
\label{firsthi}
\end{equation}

As for the Babylonian method, the Chebyshev polynomials present themselves remarkably as numerators and denominators can be expressed in function of $2T_{3^k}(X) \pm 1$. An explicit expression of the sequence is proposed in Theorem \ref{th:Hall}. After deriving a basic property of the 
Chebyshev polynomials in Lemma  \ref{lm:Hall}, a proof of Theorem \ref{th:Hall} is presented.

\begin{theorem}\label{th:Hall}

Considering the Halley's method associated to the function $f(t)=t^2-x$ and a starting point $r$, the corresponding sequence $\{{h_n}\}_{n \geq 0}$  can be expressed using the Chebyshev polynomials $T_k$ as follows:
\begin{equation}
\displaystyle (\forall n \geq 0), \;\;\;  h_n= r\prod_{i=1}^n \frac{2T_{3^{i-1}}\left(\displaystyle \frac{x+r^2}{x-r^2}\right)+1}{2T_{3^{i-1}}\left(\displaystyle \frac{x+r^2}{x-r^2}\right)-1}
\label{hallEXPL}
\end{equation}

\end{theorem}

\begin{lemma}\label{lm:Hall}
The Chebyshev polynomials verify these equalities for all $n$ greater or equal than 0:
\begin{equation}(X-1)\prod_{i=1}^n(2T_{3^{i-1}}(X)+1)^2+3(X+1)\prod_{i=1}^n(2T_{3^{i-1}}(X)-1)^2=2(2T_{3^{n}}(X)+1)\label{eqlemma0}
\end{equation}
\begin{equation}  3(X-1)\prod_{i=1}^n(2T_{3^{i-1}}(X)+1)^2+(X+1)\prod_{i=1}^n(2T_{3^{i-1}}(X)-1)^2=2(2T_{3^{n}}(X)-1)
\label{eqlemma}
\end{equation}
\end{lemma}

\begin{proof} First, the  equalities (\ref{eqlemma0}) and (\ref{eqlemma}) from Lemma \ref{lm:Hall} are proven jointly. Both initializations are immediate. The induction hypothesis consists of assuming that both equalities are verified for $n>0$. By pivoting on the system of equalities, it comes:
$$8(X-1)\prod_{i=1}^n(2T_{3^{i-1}}(X)+1)^2=8(T_{3^{n}}(X)-1)$$
$$8(X+1)\prod_{i=1}^n(2T_{3^{i-1}}(X)-1)^2=8(T_{3^{n}}(X)+1)$$

Therefore, the left-hand side of the first equality (\ref{eqlemma0}) at step $n+1$ can be written as follows:
$$(X-1)\prod_{i=1}^{n+1}(2T_{3^{i-1}}(X)+1)^2+3(X+1)\prod_{i=1}^{n+1}(2T_{3^{i-1}}(X)-1)^2$$
$$\begin{array}{llllll}&=&(T_{3^{n}}(X)-1)(2T_{3^{n}}(X)+1)^2+3(T_{3^{n}}(X)+1)(2T_{3^{n}}(X)-1)^2\\ &=&16T_{3^{n}}(X)^3-12T_{3^{n}}(X)+2 \\ 
&=&2(2(4T_{3^{n}}(X)^3-3T_{3^{n}}(X))+1) \\ &=& 2(2T_{3^{n+1}}(X) +1)\end{array}$$
This concludes the induction step for (\ref{eqlemma0})  and a quasi-identical  set of equalities can conclude the induction step for (\ref{eqlemma}).

The base case for Theorem \ref{th:Hall} is immediate from the empty-product rule. Let assume that the equality (\ref{hallEXPL}) is verified for $n>0$. As $x=M+b$, $r^2=M-b$ and $\frac{x+r^2}{x-r^2}=\frac{M}{b}$, $h_{n+1}$ can be expressed as follows:

\begin{equation}
\begin{array}{llllll}
h_{n+1}&=& \displaystyle h_n\frac{{h_{n}}^2+3(M+b)}{3{h_{n}}^2+(M+b)}\\
&=&\displaystyle h_n\frac{(M-b)\left(\prod_{i=1}^n \frac{2T_{3^{i-1}}(\frac{M}{b})+1}{2T_{3^{i-1}}( \frac{M}{b})-1}\right)^2+3(M+b)}{3(M-b)\left(\prod_{i=1}^n \frac{2T_{3^{i-1}}(\frac{M}{b})+1}{2T_{3^{i-1}}( \frac{M}{b})-1}\right)^2+(M+b)}\\
&=&\displaystyle h_n \frac{(\frac{M}{b}-1)\prod_{i=1}^n(2T_{3^{i-1}}(\frac{M}{b})+1)^2+3(\frac{M}{b}+1)\prod_{i=1}^n(2T_{3^{i-1}}(\frac{M}{b})-1)^2}{3(\frac{M}{b}-1)\prod_{i=1}^n(2T_{3^{i-1}}(\frac{M}{b})+1)^2+(\frac{M}{b}+1)\prod_{i=1}^n(2T_{3^{i-1}}(\frac{M}{b})-1)^2} \\
&=&\displaystyle h_n\frac{2(2T_{3^{n}}(\frac{M}{b})+1)}{2(2T_{3^{n}}(\frac{M}{b})-1)} \\
&=&r\displaystyle \prod_{i=1}^{n+1} \frac{2T_{3^{i-1}}\left( \frac{x+r^2}{x-r^2}\right)+1}{2T_{3^{i-1}}\left( \frac{x+r^2}{x-r^2}\right)-1}
\end{array}
\label{hn1}
\end{equation}

The penultimate equality in (\ref{hn1}) is using the results from Lemma  \ref{lm:Hall} with $X=\frac{M}{b}$. This ends the induction step and the proof of Theorem \ref{th:Hall}.

\end{proof}

A simple algorithm to compute $h_n$ is implemented in Algorithm 3 with input $x$, the initialization $r$, a non-negative index $n$ and with output $S$. Polynomial expressions of the Halley's method for inverse square roots is not derived in the current paper.

\begin{algorithm}[!ht]
\label{algohall}
\caption{$\;\;\;S=h(x,r,n)$}
\begin{algorithmic}[1]
\item $S \gets r$   \phantom{spaceenough}/*$h_0$*/
\item $T=\frac{x+r^2}{x-r^2}$
\For{i=1:n}
\item $S \gets S(1+\frac{2}{2T-1})$  \phantom{space}/*$h_i$*/
\item $T \gets T\times(4T^2-3)$     
\EndFor
\State ${\bf return} \displaystyle \;\; S$
\end{algorithmic}
\end{algorithm}

\section{Polynomial expressions in the Householder's method for square roots}

The Householder's method \cite{Householder} is an extension of both the Newton's method and the Halley's method. Considering a positive integer $d$ and the function $g(t)=\frac{1}{t^2-x}$ where $x$ is a positive real number, the Householder's method of order $d$ for square roots  is provided in (\ref{defHouse}) with initial guess $r$. This iterative algorithm is such that  the sequence $\{\mathcal{H}_{n}\}_{n \geq 0}$ converges to $\sqrt{x}$ with a rate of convergence of $d+1$.

\begin{equation}
\mathcal{H}_0=r \;\;,\;\; (\forall n \geq 0) \;\;\; \mathcal{H}_{n+1} = \mathcal{H}_n + d \frac{g^{(d-1)}(\mathcal{H}_{n})}{g^{(d)}(\mathcal{H}_{n})}
\label{defHouse}
\end{equation}

An explicit expression of the sequence is presented in Theorem \ref{th:Houseun}. After deriving a few elementary expressions in Lemma  \ref{lm:Houseun}, a proof of Theorem \ref{th:Houseun} is presented. The sequence $\mathcal{H}_{n}$ naturally depends on the order $d$, but $d$ is omitted in the notation for the sake of simplicity. We can notice that the Newton's method from (\ref{FirstNewt}) can be obtained with $d=1$ while the Halley's method from (\ref{defHall}) can be obtained with $d=2$.

\begin{theorem}\label{th:Houseun}

Considering the Householder's method of order $d$ for $\sqrt{x}$ with a  starting point $r$, the corresponding sequence $\{{\mathcal{H}_{n}}\}_{n \geq 0}$  can be expressed as follows:
\begin{equation}
\mathcal{H}_0=r \;\;,\;\; (\forall n \geq 0) \;\;\; \mathcal{H}_{n+1} = \frac{\displaystyle  \sum_{k=0}^{\myceil{\frac{d}{2}}} {{d+1}\choose{2k}} \mathcal{H}_{n}^{d+1-2k}x^k }{\displaystyle  \sum_{k=0}^{\myfloor{\frac{d}{2}}} {{d+1}\choose{1+2k}} \mathcal{H}_{n}^{d-2k}x^k }
\label{defHouse2}
\end{equation}
\end{theorem}

\begin{lemma}\label{lm:Houseun}
Considering the function $G(t)=\frac{1}{t^2-1}$, its $p$-th derivative is as follows:
\begin{equation} 
G^{(p)}(t)=\frac{(-1)^p p!}{2(t^2-1)^{p+1}} \displaystyle ((t+1)^{p+1} -(t-1)^{p+1})
\label{fderivn}
\end{equation}
In addition, the following equalities are verified for any real $t$ and any integer $p$:
$$\frac{1}{2}((t+1)^{p+1} -(t-1)^{p+1})=\sum_{k=0}^{\myfloor{\frac{p}{2}}} {{p+1}\choose{1+2k}} t^{p-2k}$$
\begin{equation} 
\frac{1}{2}((t+1)^{p+1} +(t-1)^{p+1})=\sum_{k=0}^{\myceil{\frac{p}{2}}} {{p+1}\choose{2k}} t^{p+1-2k}
\label{fdiff}
\end{equation}
\end{lemma}

\begin{proof} First, the expression of $G^{(p)}$ in Lemma \ref{lm:Houseun} is derived. As $G(t)=(t+1)^{-1}\times (t-1)^{-1}$, applying the Leibniz rule to $G$ allows us to obtain:

$$G^{(p)}(t)=\frac{(-1)^p p!}{(t^2-1)^{p+1}} \displaystyle \sum_{k=0}^p (t-1)^{p-k}(t+1)^k$$

The previous expression is a geometric series with common ratio $\frac{t+1}{t-1}$ an its closed-formed formula is (\ref{fderivn}). Next, both equalities in (\ref{fdiff}) can be obtained using the binomial expansion and elementary parity arguments. Finally, the expression (\ref{defHouse}) from Theorem \ref{th:Houseun} is derived. By noticing that $g(t)=\frac{1}{x} G(tx^{-1/2})$, Expression (\ref{defHouse}) can be written as follows:

$$\mathcal{H}_{n+1} = \mathcal{H}_n + d \sqrt{x} \frac{G^{(d-1)}(\mathcal{H}_{n}x^{-1/2})}{G^{(d)}(\mathcal{H}_{n}x^{-1/2})}  $$

Using successively (\ref{fderivn})  and (\ref{fdiff}) from Lemma \ref{lm:Houseun}, the previous equality becomes:

\begin{equation}
\begin{array}{llllll}
\mathcal{H}_{n+1} &=& \left. t\sqrt{x} - \sqrt{x}(t^2-1)\frac{(t+1)^d-(t-1)^d}{(t+1)^{d+1}-(t-1)^{d+1}} \right|_{t=\mathcal{H}_n x^{-1/2}} \\
&=&\left.  \sqrt{x}\frac{(t+1)^{d+1}+(t-1)^{d+1}}{(t+1)^{d+1}-(t-1)^{d+1}}    \right|_{t=\mathcal{H}_n x^{-1/2}} \\
&=& \left. 		\sqrt{x}\frac{\displaystyle  \sum_{k=0}^{\myceil{\frac{d}{2}}} {{d+1}\choose{2k}} t^{d+1-2k} }{\displaystyle  \sum_{k=0}^{\myfloor{\frac{d}{2}}} {{d+1}\choose{1+2k}} t^{d-2k} }						\right|_{t=\mathcal{H}_n x^{-1/2}} \\
&=& \frac{\displaystyle  \sum_{k=0}^{\myceil{\frac{d}{2}}} {{d+1}\choose{2k}} \mathcal{H}_{n}^{d+1-2k}x^k }{\displaystyle  \sum_{k=0}^{\myfloor{\frac{d}{2}}} {{d+1}\choose{1+2k}} \mathcal{H}_{n}^{d-2k}x^k }
\end{array}
\label{findefh}
\end{equation}

\end{proof}

Before providing an explicit expression of $\mathcal{H}_{n}$, the Chebyshev polynomials of the third kind $\{ V_n\}_{n \geq 0}$ and of the fourth kind $\{ W_n\}_{n \geq 0}$ are introduced in (\ref{TcheV}) and (\ref{TcheW}), respectively \cite{Mason20021,AGHIGH20082}.
\begin{equation} 
V_0(X)=1 \;\;,\;\; V_1(X)=2X-1 \;\;,\;\; (\forall n \geq 0) \;\; V_{n+2}(X)=2XV_{n+1}(X)-V_{n}(X)
\label{TcheV}
\end{equation}
\begin{equation} 
W_0(X)=1 \;\;,\;\; W_1(X)=2X+1 \;\;,\;\; (\forall n \geq 0) \;\; W_{n+2}(X)=2XW_{n+1}(X)-W_{n}(X)
\label{TcheW}
\end{equation}

The expression of $\mathcal{H}_{n}$ is presented in  Theorem \ref{th:HouseZE} either from Chebyshev polynomials or from monomials while distinguishing whether the order is odd or even. A proof of Theorem \ref{th:HouseZE} follows.

\begin{theorem}\label{th:HouseZE}
Considering the Householder's method of order $d$ for $\sqrt{x}$ with a  starting point $r$, the corresponding sequence $\{{\mathcal{H}_{n}}\}_{n \geq 0}$  can be expressed in terms of Chebyshev polynomials as follows:

\begin{enumerate}
 \item If $d$ is even and $n$ greater or equal than 0, we obtain:
\begin{equation} 
\mathcal{H}_{n}=\displaystyle r  \left. \frac{\displaystyle W_{\textstyle \frac{(d+1)^n-1}{2}}\left(X\right)}{\displaystyle V_{ \textstyle \frac{(d+1)^n-1}{2}}\left(X\right)}\right|_{X=\left(\frac{x+r^2}{x-r^2}\right)}
\label{HNPAIR}
\end{equation}
\item If $d$ is odd  and $n$ greater or equal than 1, we obtain:
\begin{equation} 
\mathcal{H}_{n}=  \left. \displaystyle   r\frac{\displaystyle T_{ \textstyle \frac{(d+1)^n}{2}}\left(X\right)}{\displaystyle(X-1) U_{ \textstyle \frac{(d+1)^n}{2}-1}\left(X\right)}\right|_{X=\left(\frac{x+r^2}{x-r^2}\right)}
\label{HNIMPAIR}
\end{equation}

\end{enumerate}

Furthermore, the sequence $\{\mathcal{H}_{n}\}_{n \geq 0}$ can be expressed in function of monomials of $X=\left(\frac{x+r^2}{x-r^2}\right)$ such as follows:

\begin{enumerate}
 \item If $d$ is even:
\begin{equation} 
\mathcal{H}_{n}= \displaystyle r  \left. \displaystyle \prod_{k=1}^{\frac{(d+1)^n-1}{2}} \frac{\displaystyle X-cos\left( \frac{2 k \pi}{(d+1)^n} \right)}{\displaystyle X-cos\left( \frac{(2 k-1) \pi}{(d+1)^n} \right)} \right|_{X=\left(\frac{x+r^2}{x-r^2}\right)}
\label{HNPAIRMON}
\end{equation}
\item If $d$ is odd: 
\begin{equation} 
\mathcal{H}_{n}= \displaystyle r  \left. \displaystyle \prod_{k=0}^{\frac{(d+1)^n}{2}-1} \frac{\displaystyle X-cos\left( \frac{(2 k+1) \pi}{(d+1)^n} \right)}{\displaystyle X-cos\left( \frac{2 k \pi}{(d+1)^n} \right)} \right|_{X=\left(\frac{x+r^2}{x-r^2}\right)}
\label{HNIMPAIRMON}
\end{equation}
\end{enumerate}
\end{theorem}

\begin{proof} The case for even $d=2p$ is first considered. Let us introduce the two following rational functions $K_e(X)= \frac{ W_{ \frac{(d+1)^n-1}{2}}\left(X\right)}{ V_{  \frac{(d+1)^n-1}{2}}\left(X\right)}$ and $L_e(X)=\prod_{k=1}^{\frac{(d+1)^n-1}{2}} \frac{ X-cos\left( \frac{2 k \pi}{(d+1)^n} \right)}{ X-cos\left( \frac{(2 k-1) \pi}{(d+1)^n} \right)}$.  Given an integer $k$, the degree of $V_k$ and $W_k$ is $k$, their leading coefficient is  $2^k$, and their roots are $\{\frac{2j \pi}{2k+1}\}_{j=1..k}$ and  $\{\frac{(2j-1) \pi}{2k+1}\}_{j=1..k}$, respectively \cite{Mason20021}. Therefore $K_e$ and $L_e$ are  identical and it implies that  the expressions (\ref{HNPAIR}) and (\ref{HNPAIRMON}) are the same.
Next, we evaluate $W_{ \frac{(d+1)^n-1}{2}}$ and $V_{ \frac{(d+1)^n-1}{2}}$
at $cos\left( \frac{2 k \pi}{(d+1)^{n+1}} \right)$ and $cos\left( \frac{(2 k-1) \pi}{(d+1)^{n+1}} \right)$ using the identities $V_k(cos(\theta))=\frac{cos((2k+1)\theta/2)}{cos(\theta/2)}$ and $W_k(cos(\theta))=\frac{sin((2k+1)\theta/2)}{sin(\theta/2)}$ \cite{Mason20021}:
\begin{equation}
\begin{array}{llllll}
\displaystyle W_{\textstyle \frac{(d+1)^n-1}{2}}\left(cos\left( \frac{2 k \pi}{(d+1)^{n+1}} \right)\right)&=&  \frac{\displaystyle sin\left(\frac{k \pi}{d+1}\right)}{\displaystyle sin \left( \frac{ k \pi}{(d+1)^{n+1}}\right)} \;\;,\;\;\\
\displaystyle V_{\textstyle \frac{(d+1)^n-1}{2}}\left(cos\left( \frac{2 k \pi}{(d+1)^{n+1}} \right)\right) &=&  \frac{\displaystyle cos\left( \frac{ k \pi}{(d+1)}\right)}{\displaystyle cos \left(\frac{k \pi}{(d+1)^{n+1}}\right)}\\
\displaystyle W_{\textstyle \frac{(d+1)^n-1}{2}}\left(cos\left( \frac{(2 k-1) \pi}{(d+1)^{n+1}} \right)\right) &=& \frac{\displaystyle sin\left(\frac{(2k-1) \pi}{2(d+1)}\right)}{\displaystyle sin \left( \frac{ (2k-1) \pi}{2(d+1)^{n+1}}\right)} \;\;,\;\; \\
\displaystyle V_{\textstyle \frac{(d+1)^n-1}{2}}\left(cos\left( \frac{(2 k-1) \pi}{(d+1)^{n+1}} \right)\right) &=& \frac{\displaystyle cos\left( \frac{ (2k-1) \pi}{2(d+1)}\right)}{\displaystyle cos \left(\frac{(2k-1) \pi}{2(d+1)^{n+1}}\right)} \\
\end{array}
\label{wprootnp1}
\end{equation}

We prove now by induction that the expressions (\ref{defHouse2}) and (\ref{HNPAIR}) are identical. The base case for $n=0$ is immediate from the empty-product rule. Let us assume that both expressions are identical for $n>0$. Let us introduce a rational function $M_e$ as follows:
\begin{equation}M_e(X)= K_e(X)\frac{\displaystyle  \sum_{j=0}^{p} {{2p+1}\choose{2j}} (X-1)^{p-j} K_e(X)^{2p-2j}(X+1)^j }{\displaystyle  \sum_{j=0}^{p} {{2p+1}\choose{1+2j}} (X-1)^{p-j} K_e(X)^{2p-2j}(X+1)^j }\label{Medeb}\end{equation}

As for $K_e$, it is straightforward to see that $M_e$ is the ratio of two polynomials of degree $\frac{(d+1)^{n+1}-1}{2}$ and that the overall leading coefficient of $M_e$ is 1.  $M_e$ can be expressed directly in terms of a ratio of polynomials of degree $\frac{(d+1)^{n+1}-1}{2}$ as follows:

$$M_e(X)= \frac{\displaystyle  \sum_{j=0}^{p} {{2p+1}\choose{2j}} (X-1)^{p-j}(X+1)^j \left[ W_{\textstyle \frac{(d+1)^n-1}{2}}\left(X\right)\right]^{2p+1-2j}\left[ V_{\textstyle \frac{(d+1)^n-1}{2}}\left(X\right)\right]^{2j} }{\displaystyle  \sum_{j=0}^{p} {{2p+1}\choose{1+2j}} (X-1)^{p-j}(X+1)^j \left[ W_{\textstyle \frac{(d+1)^n-1}{2}}\left(X\right)\right]^{2p-2j}\left[ V_{\textstyle \frac{(d+1)^n-1}{2}}\left(X\right)\right]^{2j+1} }$$

We will prove now that the roots of $M_e(X)$ are $\left\{cos\left( \frac{2 k \pi}{(d+1)^{n+1}} \right)\right \}$ while the poles of $M_e(X)$ are $\left \{cos\left( \frac{(2k-1) \pi}{(d+1)^{n+1}} \right) \right \}$ for any $k$ in $\left \{1..\frac{(d+1)^{n+1}-1}{2}\right \}$. The trigonometric identities of $cos(2\theta)\pm 1$ and the results from (\ref{wprootnp1}) are used and we obtain:

$$M_e\left(cos\left( \frac{2 k \pi}{(d+1)^{n+1}} \right)\right)=$$
\begin{equation}
\begin{array}{llllll}
\frac{\displaystyle cos\left( \frac{ k \pi}{(d+1)^{n+1}} \right)}{\displaystyle sin\left( \frac{ k \pi}{(d+1)^{n+1}} \right)} \frac{\displaystyle  \sum_{j=0}^{p} {{2p+1}\choose{2j}} (-1)^{p-j} sin \left(\frac{k \pi}{d+1}\right)^{2p+1-2j}cos\left(\frac{k \pi}{d+1}\right)^{2j} }{\displaystyle  \sum_{j=0}^{p} {{2p+1}\choose{1+2j}} (-1)^{p-j} sin \left(\frac{k \pi}{d+1} \right)^{2p-2j}cos \left(\frac{k \pi}{d+1} \right)^{2j+1} } \nonumber
\end{array}
\end{equation}
$$\displaystyle \frac{1}{M_e}\left(cos\left( \frac{(2 k-1) \pi}{(d+1)^{n+1}} \right)\right)=$$
\begin{equation}
\begin{array}{llllll}
\frac{\displaystyle sin \left(\frac{(2k-1) \pi}{2(d+1)^{n+1}} \right)}{\displaystyle cos \left(\frac{(2k-1) \pi}{2(d+1)^{n+1}} \right)} \frac{\displaystyle  \sum_{j=0}^{p} {{2p+1}\choose{1+2j}} (-1)^{p-j} sin \left(\frac{ (2k-1) \pi}{2(d+1)} \right)^{2p-2j}cos \left(\frac{ (2k-1) \pi}{2(d+1)} \right)^{2j+1} }{\displaystyle  \sum_{j=0}^{p} {{2p+1}\choose{2j}} (-1)^{p-j} sin \left(\frac{ (2k-1) \pi}{2(d+1)} \right)^{2p+1-2j}cos \left(\frac{ (2k-1) \pi}{2(d+1)} \right)^{2j} } 
\end{array}
\label{meuh}
\end{equation}

Next, we apply the binomial expansion to $\left[exp(\frac{ik \pi}{d+1})\right]^{d+1}$ and  $\left[exp(\frac{i(2k-1) \pi}{2(d+1)})\right]^{d+1}$:
$$\left[exp(\frac{ik \pi}{d+1})\right]^{d+1}=\sum_{j=0}^{2p+1}{{2p+1}\choose{j}} cos(\frac{k \pi}{d+1})^j sin(\frac{k \pi}{d+1})^{2p+1-j} i^{2p+1-j} =(-1)^{k}+ 0i$$
$$\left[exp(\frac{i(2k-1) \pi}{2(d+1)})\right]^{d+1}=\sum_{j=0}^{2p+1}{{2p+1}\choose{j}} cos(\frac{(2k-1) \pi}{2(d+1)})^j sin(\frac{(2k-1) \pi}{2(d+1)})^{2p+1-j} i^{2p+1-j} =0+(-1)^{k+1}i$$

Identifying the real and imaginary parts of the previous identities in (\ref{meuh}) , it comes automatically that $M_e\left(cos\left( \frac{2 k \pi}{(d+1)^{n+1}} \right)\right)=0$ and $\frac{1}{M_e}\left(cos\left( \frac{(2 k-1) \pi}{(d+1)^{n+1}} \right)\right)=0$. It implies that $M_e$ has the following monomial factorization: 

\begin{equation} 
M_e(X)=    \displaystyle \prod_{k=1}^{\frac{(d+1)^{n+1}-1}{2}} \frac{\displaystyle X-cos\left( \frac{2 k \pi}{(d+1)^{n+1}} \right)}{\displaystyle X-cos\left( \frac{(2 k-1) \pi}{(d+1)^{n+1}} \right)} 
\label{Mefin}
\end{equation}

Let us evaluate $r M_e(X)$ for $X=\left(\frac{x+r^2}{x-r^2}\right)$. From (\ref{Mefin}),  $r M_e(X)$ corresponds to the expression developed in  (\ref{HNPAIRMON}) at step $n+1$ or equivalently to the expression developed in  (\ref{HNPAIR}) at step $n+1$. Also $r K_e\left(\frac{x+r^2}{x-r^2}\right)$ is equal to $\mathcal{H}_{n}$. Noting that $X-1=\frac{2r^2}{x-r^2}$ and $X+1=\frac{2x}{x-r^2}$, $r M_e(X)$ from (\ref{Medeb}) can be assessed as follows:

\begin{equation}
\begin{array}{llllll}
r M_e\left(\left(\frac{x+r^2}{x-r^2}\right)\right)&=& rK_e(X)\frac{\displaystyle  \sum_{j=0}^{p} {{2p+1}\choose{2j}} \left(\frac{2r^2}{x-r^2}\right)^{p-j} K_e(X)^{2p-2j}\left(\frac{2x}{x-r^2}\right)^j }{\displaystyle  \sum_{j=0}^{p} {{2p+1}\choose{1+2j}} \left(\frac{2r^2}{x-r^2}\right)^{p-j} K_e(X)^{2p-2j}\left(\frac{2x}{x-r^2}\right)^j }\\
&=& \frac{\displaystyle  \sum_{j=0}^{p} {{2p+1}\choose{2j}} (rK_e(X))^{2p+1-2j}x^j }{\displaystyle  \sum_{j=0}^{p} {{2p+1}\choose{1+2j}} (rK_e(X))^{2p-2j}x^j } \\
&=&\frac{\displaystyle  \sum_{j=0}^{\myceil{\frac{d}{2}}} {{d+1}\choose{2j}} \mathcal{H}_{n}^{d+1-2j}x^j }{\displaystyle  \sum_{j=0}^{\myfloor{\frac{d}{2}}} {{d+1}\choose{1+2j}} \mathcal{H}_{n}^{d-2j}x^j }
\end{array}
\label{rMe}
\end{equation}

$r M_e \left(\left(\frac{x+r^2}{x-r^2}\right)\right)$ has the expression of $\mathcal{H}_{n+1}$ provided in (\ref{defHouse2}). We have proven that the expressions in (\ref{defHouse2}) and (\ref{HNPAIR}) are identical at the step $n+1$,  which ends the proof by induction when $d$ is even.
$$$$
The case for odd $d$ is now considered. Let us consider the rational function $K_o=\frac{\displaystyle T_{\frac{(d+1)^n}{2}}\left(X\right)}{\displaystyle(X-1) U_{\frac{(d+1)^n}{2}-1}\left(X\right)}$. The roots of $K_o$ are $\left\{cos\left( \frac{(2 k+1) \pi}{(d+1)^n} \right)\right\}_{k=0..(d+1)^{n}/2-1}$ and their poles are 1 and $\left\{cos\left( \frac{2 k \pi}{(d+1)^n} \right)\right\}_{k=1..(d+1)^{n}/2-1}$. Evaluating $rK_o(X)$ at $X=\left(\frac{x+r^2}{x-r^2}\right)$ proves that Equations (\ref{HNIMPAIR}) and (\ref{HNIMPAIRMON}) are identical for $n \geq 1$. By now, the remainder of the proof, i.e. establishing the equality between (\ref{HNIMPAIRMON}) and  (\ref{defHouse2}) is almost identical to the case when $d$ is even and is left to the interested reader.

\end{proof}

Additional expressions of $\left\{\mathcal{H}_{n}\right\}$ or of its residual sequence $\displaystyle \left\{\mathcal{H}_{n+1}-\mathcal{H}_{n}\right\}$ are presented in Corollary \ref{th:Corodelta}, followed by a proof.

\begin{corollary}\label{th:Corodelta}

\begin{enumerate}
 \item When $d$ is even ($d=2p$), the sequence can be expressed as a product as follows:

\begin{equation} 
\mathcal{H}_{n}= \displaystyle r  \frac{\displaystyle\prod_{j=1}^n W_p\left(T_{(2p+1)^{j-1}}\left(\frac{x+r^2}{x-r^2}\right)\right)}{\displaystyle \prod_{j=1}^n V_p \left(T_{(2p+1)^{j-1}}\left(\frac{x+r^2}{x-r^2}\right)\right)}
\label{OLDNEWpair}
\end{equation}

Furthermore, the residual sequence can be expressed as follows:

\begin{equation} 
\mathcal{H}_{n+1}-\mathcal{H}_{n}=2\mathcal{H}_{n}\frac{U_{p-1}\left(T_{(2p+1)^{n}}\left(\frac{x+r^2}{x-r^2}\right)\right)}{V_p\left(T_{(2p+1)^{n}}\left(\frac{x+r^2}{x-r^2}\right)\right)}
\label{residPAIR}
\end{equation}

\item If $d$ is odd, it can be uniquely expressed as $d=2^{\alpha+1}p+2^\alpha-1$ where $\alpha$ is an integer greater or equal than 1 and $p$ is an integer. $\alpha$ corresponds to the 2-adic valuation of $(d+1)$. $\mathcal{H}_{n}$ can be expressed as follows for $n$ greater or equal than 1:
\begin{equation} 
\mathcal{H}_{n}= \displaystyle   \frac{r}{2^{\alpha n -1}}\left. \frac{\displaystyle T_{\frac{(d+1)^n}{2}}\left(X\right)}{\displaystyle (X-1)\prod_{j=0}^{n-2}T_{\frac{(d+1)^{j+1}}{2}}\left(X\right) \displaystyle \prod_{\substack{j=0..n-1 \\ k=0..\alpha-2}} T_{2^k(d+1)^{j}}\left(X\right)\displaystyle \prod_{j=0}^{n-1} W_p\left(T_{2^\alpha(d+1)^{j}}\left(X\right)\right) }\right|_{X=\left(\frac{x+r^2}{x-r^2}\right)}
\label{HNIMPAIRTWO}
\end{equation}

Additionally, $\mathcal{H}_{n}$ can be expressed recursively:

\begin{equation} 
\mathcal{H}_{n+1}= \displaystyle   \mathcal{H}_{n} \left. \frac{T_{d+1}(X_n)}{X_n U_d(X_n)}\right|_{X_n=T_{\frac{(d+1)^n}{2}}\left(\frac{x+r^2}{x-r^2}\right)}
\label{HNIMPAIRTWOPROD}
\end{equation}

Finally, the residual sequence can be expressed as follows:
\begin{equation} 
\mathcal{H}_{n+1}-\mathcal{H}_{n}= \displaystyle   -\mathcal{H}_{n} \left. \frac{U_{d-1}(X_n)}{X_n U_d(X_n)}\right|_{X_n=T_{\frac{(d+1)^n}{2}}\left(\frac{x+r^2}{x-r^2}\right)}
\label{HNIMPAIRTWOPRODDELTA}
\end{equation}

\end{enumerate}
\end{corollary}

\begin{proof} Equation (\ref{OLDNEWpair}) appears naturally when deriving the Householder's method when $d$ is even, as we have seen for  the Halley's method $(d=2)$ in (\ref{hallEXPL}). Using their trigonometric definitions, it is easy to verify that $W_{\textstyle \frac{(d+1)^n-1}{2}}(X)=\prod_{j=1}^n W_p\left(T_{(2p+1)^{j-1}}\left(X\right)\right)$ and $V_{\textstyle \frac{(d+1)^n-1}{2}}(X)=\prod_{j=1}^n V_p\left(T_{(2p+1)^{j-1}}\left(X\right)\right)$. Therefore, Equations (\ref{HNPAIR}) and 
(\ref{OLDNEWpair}) are identical. Equation (\ref{residPAIR}) can be directly obtained from (\ref{OLDNEWpair}) using the identity $V_p-W_p=-2U_{p-1}$ \cite{Mason20021}.  Equation (\ref{HNIMPAIRTWOPROD}) can be obtained from (\ref{HNIMPAIR}) using the two classical identities $T_n \circ T_m=T_{n+m}$ and $U_{nm-1}=U_{m-1}(T_n)U_{n-1}$ \cite{Mason20021}. Equation (\ref{HNIMPAIRTWOPRODDELTA}) is a direct consequence of (\ref{HNIMPAIRTWOPROD}) using the identity $T_{d+1}(X)=XU_d(X)-U_{d-1}(X)$ \cite{Mason20021}. Equation (\ref{HNIMPAIRTWO}) appears relatively naturally when deriving the Householder's method from its definition (\ref{defHouse2}) when $d$ is odd, as for the Newton's method in (\ref{denom}). It consists of a factorization - not prime - of the Chebyshev polynomials of the second kind at the denominator. Let us introduce the  following family of polynomials $\mathcal{K}_n(X)= \displaystyle 2^{\alpha n-1}\left(\prod_{j=0}^{n-2}T_{\frac{(d+1)^{j+1}}{2}}\left(X\right) \displaystyle \prod_{\substack{j=0..n-1 \\ k=0..\alpha-2}} T_{2^k(d+1)^{j}}\left(X\right)\displaystyle \prod_{j=0}^{n-1} W_p\left(T_{2^\alpha(d+1)^{j}}\left(X\right)\right)\right)$. In order to justify (\ref{HNIMPAIRTWO}), it is sufficient to prove by induction that $\mathcal{K}_n(X)=U_{\frac{(d+1)^n}{2}-1}(X)$. For $n=1$, $\mathcal{K}_1$ reduces to $ 2^{\alpha-1}\left(\prod_{{k=0..\alpha-2}} T_{2^k}\left(X\right) W_p\left(T_{2^\alpha}\left(X\right)\right)\right)$. Using the same factorization as in (\ref{denom22}), $\mathcal{K}_1$ can be re-written as $U_{2^{\alpha-1}-1}W_p(T_{2^\alpha}(X))$, which is equal to $U_{2^\alpha p+2^{\alpha-1}-1}(X)$ using the respective definitions of $T_{2^\alpha}$, $W_p$ and $U_{2^{\alpha-1}-1}$ in terms of trigonometric functions. As $\frac{d+1}{2}-1$ is equal to $2^\alpha p+2^{\alpha-1}-1$, the base case is proven. At a step $n \geq 1$, it is assumed that $\mathcal{K}_n(X)=U_{\frac{(d+1)^n}{2}-1}(X)$. $\mathcal{K}_{n+1}(X)$ can be written as follows:
$$\mathcal{K}_{n+1}(X)=2^\alpha \prod_{k=0}^{\alpha-1}T_{2^{k-1}(d+1)^n}(X) \times  W_p(T_{2^\alpha (d+1)^n}(X)) \times U_{\frac{(d+1)^n}{2}-1}(X)$$
Using the identity $U_{2n-1}=2T_n U_{n-1}$ in a chain rule, we obtain \\ $\mathcal{K}_{n+1}(X)=U_{2^{\alpha-1}(d+1)^n-1}W_p(T_{2^\alpha (d+1)^n}(X))$. As for the base case, this product reduces to the Chebyshev polynomial of the second kind $U_{2^\alpha(d+1)^n p+2^{\alpha-1}(d+1)^n-1}(X)$. The previous polynomial is $U_{\frac{(d+1)^{n+1}}{2}-1}(X)$, which ends the proof.
\end{proof}

An algorithm to compute $\mathcal{H}_{n}$ can be derived using the properties of Corollary \ref{th:Corodelta}. It is implemented in Algorithm 4 with input $x$, the initialization $r$, the order $d$  and the positive index $n$ and with output $S$. For an efficient numerical implementation of Algorithm 4, it can be recommended to have a precise asymptotic expansion of either $\frac{2U_{\frac{d}{2}-1}(X)}{V_{\frac{d}{2}}(X)}$ or $\frac{-U_{d-1}(X)}{X U_{d}(X)}$.


\begin{algorithm}[!ht]
\label{algohall}
\caption{$\;\;\;S=H(x,r,d,n)$}
\begin{algorithmic}[1]
\item $X \gets \frac{x+r^2}{x-r^2}$
\If{$(mod(d,2)==0)$}

        \State  $S \gets r(1+\Call{Phi}{X})$ 
				\State $T \gets T_{d+1}(X)$
    
\Else
      \State $T \gets T_{\frac{d+1}{2}}(X)$ 
				\State  $S \gets \frac{rT}{(X-1) \times U_{\frac{d+1}{2}-1}(X)}$

\EndIf
\For{i=2:n}
\State $S \gets S(1+\Call{Phi}{T})$  \phantom{space}/*$H_i$*/
\State  $T \gets T_{d+1}(T)$     
\EndFor
\State ${\bf return} \displaystyle \;\; S$
\Function{Phi}{X}
\If{$(mod(d,2)==0)$}
        \State ${\bf return} \displaystyle  \frac{2U_{\frac{d}{2}-1}(X)}{V_{\frac{d}{2}}(X)}$
    
\Else
     \State ${\bf return} \displaystyle  \frac{-U_{d-1}(X)}{X U_{d}(X)}$
\EndIf
\EndFunction
\end{algorithmic}
\end{algorithm}

Finally, it is possible to express $\mathcal{H}_{n}$ without the use of trigonometric-related functions as expressed in Theorem \ref{th:sqrtxR}. A proof follows.

\begin{theorem}\label{th:sqrtxR}
Considering the Householder's method of order $d$ for $\sqrt{x}$ with a  starting point $r$, the corresponding sequence $\{{\mathcal{H}_{n}}\}_{n \geq 0}$  can be expressed  as follows:

\begin{equation} 
\mathcal{H}_{n}= \displaystyle  \sqrt{x} \frac{\displaystyle  \left(r +\sqrt{x}\right)^{(d+1)^n}+\left(r -\sqrt{x}\right)^{(d+1)^n}}{\displaystyle  \left(r +\sqrt{x}\right)^{(d+1)^n}-\left(r -\sqrt{x}\right)^{(d+1)^n}}
\label{sqrt1}
\end{equation}

The previous expression can be re-written as a rational function of $(x,r)$ as follows:

\begin{equation}
\mathcal{H}_{n} = r\frac{\displaystyle  \sum_{k=0}^{\myceil{\frac{(d+1)^n-1}{2}}} {{(d+1)^n}\choose{2k}} x^kr^{(d+1)^n -2k} }{\displaystyle  \sum_{k=0}^{\myfloor{\frac{(d+1)^n-1}{2}}} {{(d+1)^n}\choose{1+2k}} x^kr^{(d+1)^n-2k} }
\label{sqrt2}
\end{equation}

\end{theorem}

\begin{proof}
We prove by induction that Equations (\ref{defHouse2}) and  (\ref{sqrt1}) are identical. For the initialization, both equations are equal to $r$. Let us assume the equality holds at step $n \geq 0$. Let us denote $\phi_n=\left(r +\sqrt{x}\right)^{(d+1)^n}$  and $\psi_n=\left(r-\sqrt{x}\right)^{(d+1)^n}$. Therefore $\mathcal{H}_{n}=\sqrt{x}\frac{\phi_n+\psi_n}{\phi_n-\psi_n}$. The induction step can be obtained using the same method as in (\ref{findefh}):

\begin{equation}
\begin{array}{llllll}
\displaystyle  \sqrt{x} \frac{\displaystyle  \left(r +\sqrt{x}\right)^{(d+1)^{n+1}}+\left(r -\sqrt{x}\right)^{(d+1)^{n+1}}}{\displaystyle  \left(r +\sqrt{x}\right)^{(d+1)^{n+1}}-\left(r -\sqrt{x}\right)^{(d+1)^{n+1}}} &=& \sqrt{x} \frac{\displaystyle \phi_n^{d+1}+\psi_n^{d+1}}{\displaystyle \phi_n^{d+1}-\psi_n^{d+1}} \\
&=& 
\sqrt{x} \frac{\displaystyle \left(\phi_n +\psi_n +\phi_n-\psi_n \right)^{d+1}+\left(\phi_n +\psi_n -(\phi_n-\psi_n) \right)^{d+1}}{\displaystyle \left(\phi_n +\psi_n +\phi_n-\psi_n \right)^{d+1}-\left(\phi_n +\psi_n -(\phi_n-\psi_n) \right)^{d+1}} \\
&=&  \sqrt{x}\frac{\displaystyle \left(\left(\frac{\phi_n+\psi_n}{\phi_n-\psi_n} \right)+1\right)^{d+1}+\left(\left(\frac{\phi_n+\psi_n}{\phi_n-\psi_n} \right)-1\right)^{d+1}}{\displaystyle\left(\left(\frac{\phi_n+\psi_n}{\phi_n-\psi_n} \right)+1\right)^{d+1}-\left(\left(\frac{\phi_n+\psi_n}{\phi_n-\psi_n} \right)-1\right)^{d+1}}  \\
&=&\left.  \sqrt{x}\frac{ \displaystyle (t+1)^{d+1}+(t-1)^{d+1}}{ \displaystyle (t+1)^{d+1}-(t-1)^{d+1}}    \right|_{t=\mathcal{H}_n x^{-1/2}} \\
&=&\mathcal{H}_{n+1}
\end{array}
\label{sqrtproof}
\end{equation}

 Equation  (\ref{sqrt1}) highlights that the $(n+1)$-th term of the Householder's sequence of order $d$ is equal to the second term of the Householder's sequence of order $(d+1)^n-1$. Equation (\ref{sqrt2}) is simply obtained from (\ref{defHouse2}) for $n=0$ and the order $(d+1)^n-1$. Similar functions to the ones in Equation (\ref{sqrt2}) have been recently studied in \cite{Lima} and are related to the tangent analog of the Chebyshev polynomials.

\end{proof}

We can notice that the sequence $\left\{\mathcal{A}_n\right\}_{n \geq 1}$ defined by $\mathcal{A}_n=\sqrt{x} \frac{\displaystyle  \left(r +\sqrt{x}\right)^{n}+\left(r -\sqrt{x}\right)^{n}}{\displaystyle  \left(r +\sqrt{x}\right)^{n}-\left(r -\sqrt{x}\right)^{n}}$ is among the slowest sequence to converge to $\sqrt{x}$ while at the same the time  it is featuring in its subsequences all the sequences of  the Householder's method  for $\sqrt{x}$ at every order.  This sequence has already been obtained by Yeyios \cite{YEYIOS} from continued fraction expansions. The Newton's method and more generally the Householder's method for square roots of integer numbers is intimately related to Pell's Equations \cite{BarbeauBOOKPell,McBride99}.

\section{A note on the Householder's method for nth roots }

Algorithms for the nth root computation have already been developed \cite{DUBEAU200966}.  An introduction to the Householder's method to obtain $\sqrt[p]{x}$  is now discussed. Considering an integer $p$, an order $d$ and the function $g_p(t)=\frac{1}{t^p-x}$ where $x$ is a positive real number, the Householder's method of order $d$ for pth root  is provided in (\ref{defHousep}) with initial guess $r$.
As for the method developed in Section 6, the sequence $\{\bold{H}_{n}\}_{n \geq 0}$ converges to $\sqrt[p]{x}$ with a rate of convergence of $d+1$.

\begin{equation}
\bold{H}_0=r \;\;,\;\; (\forall n \geq 0) \;\;\; \bold{H}_{n+1} = \bold{H}_n + d \frac{g_p^{(d-1)}(\bold{H}_{n})}{g_p^{(d)}(\bold{H}_{n})}
\label{defHousep}
\end{equation}

Numerous expressions in the Householder's method for square roots are based on binomial coefficients. Considering the nth root extraction, we need to introduce the generalized binomial coefficients of order $p$ ${{n}\choose{m}}_p$  \cite{Bollinger,Bondarenko} which naturally appear in the development of $B_p(x)=\left(\sum_{k=0}^{p-1}x^k\right)$ at the power $n$ as follows:

\begin{equation}
B_p(x)^n=\left(1+x+\ldots+x^{p-1}\right)^n =\sum_{m=0}^{(p-1)n} {{n}\choose{m}}_p x^m
\label{genbinom}
\end{equation}

The generalized binomial coefficients of order $p$ can be obtained from the binomial coefficients as follows \cite{Bondarenko}:
$${{n}\choose{m}}_p=\sum_{k=0}^{\myfloor{\frac{m}{p}}}(-1)^k{{n}\choose{k}}{{n+m-pk-1}\choose{n-1}}$$

In addition, the parity arguments are extended to their modular counterparts involving series multisection \cite{BalaichOndrus,Beauregard}. Given a function $f$ and a radical $p$, let us denote the primitive root of unity $\xi=exp(\frac{2i \pi}{p})$ and we introduce $p$ functions $\{\Big[ f\Big]_\ell (t)\}_{\ell=0..p-1}$ as follows:

\begin{equation}
\Big[ f\Big]_\ell (t) =\frac{1}{p}\sum_{k=0}^{p-1} \xi^{-\ell k}f(\xi^k t)
\label{fkmodular}
\end{equation}

It is straightforward to verify that if $f$  has  a power expansion of the type $\sum_{n \geq 0} a_n x^n$, $\Big[ f\Big]_\ell$ is expressed as $\sum_{n \geq 0} a_{pn+\ell} x^{pn+\ell}$. The functions  are commonly referred as Roots of Unity Filters and $f$ can be reconstructed from their sums.

An explicit expression of the sequence $\{\bold{H}_{n}\}_{n \geq 0}$ is presented in Theorem \ref{th:Houseunpp}. A proof follows. 

\begin{theorem}\label{th:Houseunpp}

Considering the Householder's method of order $d$ for $\sqrt[p]{x}$ with a  starting point $r$, the corresponding sequence $\{{\bold{H}_{n}}\}_{n \geq 0}$  can be expressed as follows:
\begin{equation}
\bold{H}_0=r \;\;,\;\; (\forall n \geq 0) \;\;\; \bold{H}_{n+1} = \sqrt[p]{x} \frac{\Big[B_p^{d+1}\Big]_{1-d[p]}\left(\bold{H}_{n}x^{-1/p}\right)}{\Big[B_p^{d+1}\Big]_{-d[p]}\left(\bold{H}_{n}x^{-1/p}\right)}
\label{defHouse2pp}
\end{equation}

The previous expression can be formulated as follows:

\begin{equation}
\bold{H}_0=r \;\;,\;\; (\forall n \geq 0) \;\;\; \bold{H}_{n+1} = \frac{\displaystyle \sum_{k=0}^{\myfloor{\frac{(p-1)d+1}{p}}} {{d+1}\choose{p(k+1)-2}}_p \bold{H}_{n}^{d(p-1)+1-pk}x^{k}}{\displaystyle \sum_{k=0}^{\myfloor{\frac{(p-1)d}{p}}} {{d+1}\choose{p(k+1)-1}}_p \bold{H}_{n}^{d(p-1)-pk}x^k}
\label{defHouse2pp2}
\end{equation}
\end{theorem}

\begin{proof} \label{proof:Housepp1}

The partial fraction decomposition of  $\frac{1}{t^p-1}$ is an elementary result:
$$ \frac{1}{t^p-1}=\frac{1}{p}\sum_{k=0}^{p-1} \frac{\xi^k}{t-\xi^k}$$
Therefore its $r$-th derivative has the following expression:
$$\frac{\mathrm{d}^r \frac{1}{t^p-1}}{\mathrm{d}t^r}=\frac{(-1)^r r!}{p}\sum_{k=0}^{p-1} \frac{\xi^k}{(t-\xi^k)^{r+1}}$$

By noticing that $g_p(t)=\frac{1}{x}\frac{1}{(tx^{-1/p})^p-1}$, Expression (\ref{defHousep}) can be written as follows:

$$\bold{H}_{n+1} = \sqrt[p]{x} \left. \left[ t-  \frac{\sum_{k=0}^{p-1} \frac{\xi^k}{(t-\xi^k)^{d}}}{\sum_{k=0}^{p-1} \frac{\xi^k}{(t-\xi^k)^{d+1}}} \right]\right|_{t=\bold{H}_{n}x^{-1/p}} $$

We can now notice that $\frac{t\xi^k}{(t-\xi^k)^{d+1}}=\frac{\xi^k}{(t-\xi^k)^{d}}+\frac{\xi^{2k}}{(t-\xi^k)^{d+1}}$, which implies:

\begin{equation}
\bold{H}_{n+1} = \sqrt[p]{x} \left. \left[\frac{ \displaystyle \sum_{k=0}^{p-1} \frac{\xi^{2k}}{(t-\xi^k)^{d+1}}}{ \displaystyle \sum_{k=0}^{p-1} \frac{\xi^k}{(t-\xi^k)^{d+1}}} \right]\right|_{t=\bold{H}_{n}x^{-1/p}}
\label{BoldHn}
\end{equation}

Let us now derive the function $\frac{\Big[B_p^{d+1}\Big]_{l}(t)}{(t^p-1)^{d+1}}$ using the identity $t^{p}-1=(t-1)B(t)$ and the series multisection defined in (\ref{fkmodular}) as follows :

\begin{equation}
\begin{array}{llllll}
\frac{\Big[B_p^{d+1}\Big]_{\ell}(t)}{(t^p-1)^{d+1}} &=& \frac{1}{p} \frac{ \displaystyle \sum_{k=0}^{p-1} \xi^{-\ell k}B(\xi^k t)^{d+1}}{  \displaystyle (t^p-1)^{d+1}} = \frac{1}{p} \displaystyle \sum_{k=0}^{p-1} \frac{\xi^{-\ell k}}{(\xi^kt-1)^{d+1}} \\
&=& \frac{1}{p} \displaystyle \sum_{k=0}^{p-1} \frac{\xi^{-\ell k}}{\xi^{k(d+1)}(t-\xi^{p-k})^{d+1}} \\
&=& \frac{1}{p} \displaystyle \sum_{k=0}^{p-1} \frac{\xi^{k(\ell+d+1)}}{(t-\xi^{k})^{d+1}} \\
\end{array}
\label{BLsurRP}
\end{equation}

The previous function can match both the numerator and the denominator of (\ref{BoldHn}). For the numerator, we need to identify $\ell$ such that $\ell+d+1=2[p]$ and for the denominator  $\ell$ such that $\ell+d+1=1[p]$. Equation (\ref{defHouse2pp}) ensues. To obtain (\ref{defHouse2pp2}), rather than using the explicit expression of $\Big[B_p^{d+1}\Big]_{\ell}(t)$ using (\ref{fkmodular}), we use the fact that it filters all powers except the ones congruent to $\ell$ modulo $p$. The highest degree of the numerator is $d(p-1)+1$ while the highest degree of the denominator is  $d(p-1)$. Therefore, we obtain:

\begin{equation}
\begin{array}{llllll}
\bold{H}_{n+1} &=& \sqrt[p]{x} \frac{\displaystyle \Big[B_p^{d+1}\Big]_{1-d[p]}\left(\bold{H}_{n}x^{-1/p}\right)}{\displaystyle \Big[B_p^{d+1}\Big]_{-d[p]}\left(\bold{H}_{n}x^{-1/p}\right)} = \left. \sqrt[p]{x}  \frac{\displaystyle \sum_{\stackrel{m=0}{m+d-1=0[p]}}^{(p-1)(d+1)} {{d+1}\choose{m}}_p t^m}{\displaystyle \sum_{\stackrel{m=0}{m+d=0[p]}}^{(p-1)(d+1)} {{d+1}\choose{m}}_p t^m} \right|_{t=\bold{H}_{n}x^{-1/p}} \\
&=& \left. \sqrt[p]{x}  \frac{\displaystyle \sum_{k=0}^{\myfloor{\frac{(p-1)d+1}{p}}} {{d+1}\choose{d(p-1)+1-pk}}_p t^{d(p-1)+1-pk}}{\displaystyle \sum_{k=0}^{\myfloor{\frac{(p-1)d}{p}}} {{d+1}\choose{d(p-1)-pk}}_p t^{d(p-1)-pk}} \right|_{t=\bold{H}_{n}x^{-1/p}}\\
&=& \sqrt[p]{x}  \frac{\displaystyle \sum_{k=0}^{\myfloor{\frac{(p-1)d+1}{p}}} {{d+1}\choose{d(p-1)+1-pk}}_p \bold{H}_{n}^{d(p-1)+1-pk}x^{k-1/p}}{\displaystyle \sum_{k=0}^{\myfloor{\frac{(p-1)d}{p}}} {{d+1}\choose{d(p-1)-pk}}_p \bold{H}_{n}^{d(p-1)-pk}x^k}\\
\end{array}
\label{HNppP1}
\end{equation}

The radical values collapses and we obtain (\ref{defHouse2pp2}) using the identity ${{n}\choose{m}}_s={{n}\choose{(s-1)n-m}}_s$  \cite{Bondarenko}, which ends the proof.

\end{proof}

\section{A note on the asymptotic expansion of Chebyshev functions \label{sec:extension}}

When considering the Householder's method of order $d$ for $\sqrt{x}$, Algorithm 4 identified the need to have a precise asymptotic evaluation 
 of $\frac{U_{d-1}(X)}{X U_{d}(X)}$ and $\frac{2U_{\frac{d}{2}-1}(X)}{V_{\frac{d}{2}}(X)}$ when $d$ is odd and even, respectively. The goal of this section is to provide an asymptotic expansion of these functions and to discuss the presence of these families of functions in the study of lattice paths.

Let us first denote $f_d(x)=\frac{U_{d-1}(x)}{x U_{d}(x)}$ for  $d \geq 1$. It is straightforward from (\ref{defUn}) to obtain the following recurrence relation:

\begin{equation}
f_1(x)=\frac{1}{2x^2} \;\;\;,\;\;\;\; f_d(x)=\frac{1}{x^2(2-f_{d-1}(x))}
\label{recfd}
\end{equation}

From (\ref{recfd}), we can in particular explicit the pointwise limit of $f_d(x)$ for $x>1$ which is $f(x)=1-\sqrt{1-\frac{1}{x^2}}$. 
Before expressing $f_d$ as a power series, we define one of the most common lattice path called Dyck path \cite{Krattenthaler}. A Dyck path of semilength $n$ and of maximum height $h$ is a lattice walk from $(0,0)$ to $(2n,0)$ with steps of the form $(1,1)$ and $(1,-1)$ with a height bounded in the interval $[0,h]$. The number of Dyck paths of semilength $n$ and of maximum height $h$ is denoted $\Delta_{n,h}$. We can observe that $\Delta_{4,1}=1$, $\Delta_{4,2}=8$, $\Delta_{4,3}=13$ and $\Delta_{4,k}=14$ for $k\geq 4$. The natural initialization of the sequence is $\{\Delta_{0,h}=1\}_{h \geq 0}$ and $\{\Delta_{n,0}=0\}_{n \geq 1}$. Enumerating Dyck paths can be also found in ballot counting problems \cite{Bertrand}, plane trees \cite{Bruijn} or permutations \cite{Krattenthaler}. It is well established that $\Delta_{n,h}$ is equal to the Catalan number $C_n=\frac{{{2n}\choose{n}} }{n+1}$ when $h\geq n$ \cite{Aigner2001}.

An early reference to the sequence $\left\{\Delta_{n,h}\right\}_{\{n,h \geq 0\}}$ can be traced back in the work of Kreweras \cite{Kreweras} using Fibonacci polynomials, which are closely related to the Chebyshev polynomials of the second kind. The relationship between Dyck paths and the ratio of Chebyshev polynomials has been further established in  \cite{CHOW1999119,GuibertMansour,Krattenthaler}, usually involving generating functions. Theorem \ref{th:Dyck1} presents an asymptotic expansion of $f_d$ expressed as a power series, followed by a proof.

\begin{theorem}\label{th:Dyck1}

We consider the family of functions $\{f_d(x)=\frac{U_{d-1}(x)}{x U_{d}(x)}\}_{d \geq 1}$ over the interval $[1,\infty[$ and $\Delta_{n,h}$ as the number of Dyck paths of semilength $n$ and of maximum height $h$. $f_d$ can be expressed using power series as follows:

\begin{equation}
f_d(x)=\sum_{i=0}^\infty \frac{\Delta_{i,d-1}}{2^{2i+1}x^{2i+2}}
\label{fdth}
\end{equation}

\end{theorem}

\begin{proof}

Equation (\ref{fdth}) is correct for $d=1$, based on (\ref{recfd}). Considering Dyck paths of semilength $n$ and maximum height $h$, a conditioning with respect to the last return to the x-axis leads to \cite{Bruijn}:
\begin{equation}
\Delta_{n,h}=\sum_{k=0}^{n-1} \Delta_{k,h}\Delta_{n-1-k,h-1}
\label{deltadecomp}
\end{equation}
For $d\geq 2$, given the parity  of $f_d$ highlighted in (\ref{recfd}), $f_d$ can be expressed as $\sum_{i=0}^\infty \frac{\alpha_{i,d}}{x^{2i+2}}$. Rewritting (\ref{recfd}) as $2x^2f_d(x)-x^2f_d(x)f_{d-1}(x)=1$ and using the Cauchy product, it leads to:

\begin{equation} \left\{    \begin{array}{lllllllllll} 2\alpha_{0,d}&=&1 \\      2\alpha_{i,d}&=&\sum_{k=0}^i \alpha_{k,d}\alpha_{i-k,d-1} \;\;\; (i\geq 1)\end{array}\right.
\label{recalpha}
\end{equation}
Finally, $\alpha_{i,d}=\frac{\Delta_{i,d-1}}{2^{2i+1}}$ is the appropriate candidate for both the initialization and the general case in (\ref{recalpha}), which ends the proof.

\end{proof}

The approximation of $f_d$ through the sequence $\left\{\Delta_{n,d-1}\right\}_{\{n \geq 0\}}$ can be obtained in multiple ways. The sequence $\left\{\Delta_{n,h}\right\}_{\{n,h \geq 0\}}$ corresponds to the OEIS sequence {A080934} \cite{OEIS} and the recurrence relation presented in (\ref{deltadecomp}) \cite{Bruijn} is a common identity to compute the sequence. Closed form expressions are presented in \cite{HEIN2022112711} and the sequence has also been studied through its differences in \cite{Kreweras}.

Let us now study the family of functions $g_d=\frac{2U_{\frac{d}{2}-1}(X)}{V_{\frac{d}{2}}(X)}$. As for $f_d$, it is interesting to study $g_d$
for odd and even values of $d$. When $d$ is odd, it is necessary to define  Chebyshev functions of half-integer order, which has been comonly studied over the interval [-1,1] \cite{Ricci}. As we evaluate $g_d$ over the interval $[1,\infty[$ in Algorithm 4, the Chebyshev functions of half integer order are proposed in Lemma \ref{lmPChe}
over the interval $[1,\infty[$. The proof is simply obtained from the hyperbolic definition of the Chebyshev polynomial over the interval $[1,\infty[$, see for example \cite{Mason20021}.

\begin{lemma}\label{lmPChe}
The Chebyshev polynomials of the four kinds can be defined at half-integer orders using the following identities for $p \geq 0$:
\begin{equation}
 U_{p-\frac{1}{2}}=\sqrt{\frac{1}{2(1+z)}} W_p(z)
\label{CHEUW}
\end{equation}
\begin{equation}
 T_{p+\frac{1}{2}}=\sqrt{\frac{1+z}{2}} V_p(z) 
\label{CHETV}
\end{equation}
\end{lemma}

Based on the identities of Lemma \ref{lmPChe},  the family of functions $g_d$ is defined as follows:

\begin{equation} g_d(x)=\left\{    \begin{array}{lllllllllll} \frac{2U_{\frac{d}{2}-1}(x)}{V_{\frac{d}{2}}(x)} & \text{if   $d$ is even } \\     \frac{W_{\frac{d-1}{2}}(x)}{T_{\frac{d+1}{2}}(x)} & \text{if   $d$ is odd }\end{array}\right.
\label{gddef}
\end{equation}

In order to express $g_d$ as a power series, Lemma \ref{lmGDFD} presents an alternate expression of $g_d$ along with a recurrence relation involving both $f_d$ and $g_d$. A proof follows.

\begin{lemma}\label{lmGDFD}
The family of functions $\{g_d(x)\}_{d \geq 1}$ can be expressed using only Chebyshev polynomials of the second kind as follows:
\begin{equation}
g_d(x)=2 \frac{\displaystyle \sum_{k=0}^{d-1} U_k(x)}{U_d(x)}
\label{CHEGU1}
\end{equation}
In addition, $f_d$ and $g_d$ obey to the following recurrence relation:
\begin{equation}
g_1(x)=\frac{1}{x} \;\;\;\;,\;\; g_d(x)=x f_d(x)(2+g_{d-1}(x)) \;\;(d \geq 1)
\label{CHEGU2}
\end{equation}
\end{lemma}

\begin{proof}

The proof of establishing Equation (\ref{CHEGU1}) is mainly based on Lagrange's trigonometric identity \\ $\sum_{k=0}^n sin(k\theta)=\frac{sin(\frac{(n+1)\theta}{2})sin(\frac{n \theta}{2})}{sin(\frac{\theta}{2})}$. Therefore we obtain:

\begin{equation}
2 \frac{\displaystyle \sum_{k=0}^{d-1} U_k(cos(\theta))}{U_d(cos(\theta))} =\frac{2sin(\frac{(d+1)\theta}{2})sin(\frac{d \theta}{2})}{sin((d+1)\theta)sin(\frac{\theta}{2})}=\frac{sin(\frac{d \theta}{2})}{cos(\frac{(d+1)\theta}{2})sin(\frac{\theta}{2})}
\label{CHEGU3}
\end{equation}

Using the trigonometric definitions of the Chebyshev polynomials of the four kinds \cite{Mason20021} and based on the parity of $d$, Equation (\ref{CHEGU1}) can be established over the interval [-1,1], except for the isolated singularities. The equality can be extended to $[1,\infty[$ by analytic continuation. 
Equation (\ref{CHEGU2}) is a direct consequence of Equation (\ref{CHEGU1}). 

\end{proof}

From (\ref{CHEGU2}), it is possible to  derive the pointwise limit of $g_d(x)$ for $x>1$ which is $g(x)=\sqrt{\frac{x+1}{x-1}}-1$. 
Before expressing $g_d$ as a power series, we introduce the concept of Symmetric Dyck path \cite{DengDeng,Broadway,Lipton}. A Symmetric Dyck path of semilength $n$ and of maximum height $h$ is  a Dyck path of semilength $n$ and of maximum height $h$ which is symmetrical from the line $x=n$. The number of Symmetric Dyck paths of semilength $n$ and of maximum height $h$ is denoted $\Delta^S_{n,h}$. From Figure 1,  We obtain $\Delta^S_{4,1}=1$, $\Delta^S_{4,2}=4$, $\Delta^S_{4,3}=5$ and $\Delta^S_{4,k}=6$ for $k\geq 4$. The initialization of the sequence is similar to the previous sequence with $\{\Delta^S_{0,h}=1\}_{h \geq 0}$ and $\{\Delta^S_{n,0}=0\}_{n \geq 1}$. It is well known \cite{Broadway} that $\Delta^S_{n,h}$ is equal to the central binomial coefficient $D_n={{n}\choose{\myfloor{\frac{n}{2}}}}$ when $h \geq n$ \cite{Ferrari}. A Symmetric Dyck path of semilength $n$ and of maximum height $h$ can be commonly decomposed in the following way: given $k \leq \myfloor{\frac{n}{2}}$, a path is composed of a Dyck path of semilength $k$ and of maximum height $h$, a step $(1,1)$, a (shifted) Symmetric Dyck path of semilength $n-1-2k$ and of maximum height $h-1$, a step $(1,-1)$ and finally the symmetric of the initial Dyck path of semilength $k$. The notable exception is the existence, when $n$ is even, of Symmetric Dyck paths of semilength $n$ and of maximum height $h$ composed of two symmetric Dyck paths of semilength $\myfloor{\frac{n}{2}}$. Given the convention $\{\Delta^S_{-1,h}=1\}_{h \geq 0}$ and conditioning with respect to the last return to the x-axis before $n$, we finally obtain the following identity \cite{Broadway,GuibertMansour}:

\begin{equation}
\Delta^S_{n,h}= \sum_{k=0}^{\myfloor{\frac{n}{2}}} \Delta_{k,h}\Delta^S_{n-1-2k,h-1}
\label{Deltasdelta}
\end{equation}

\begin{figure}[!h]
    \begin{center}
    \includegraphics[height= 5 cm,width=14 cm]{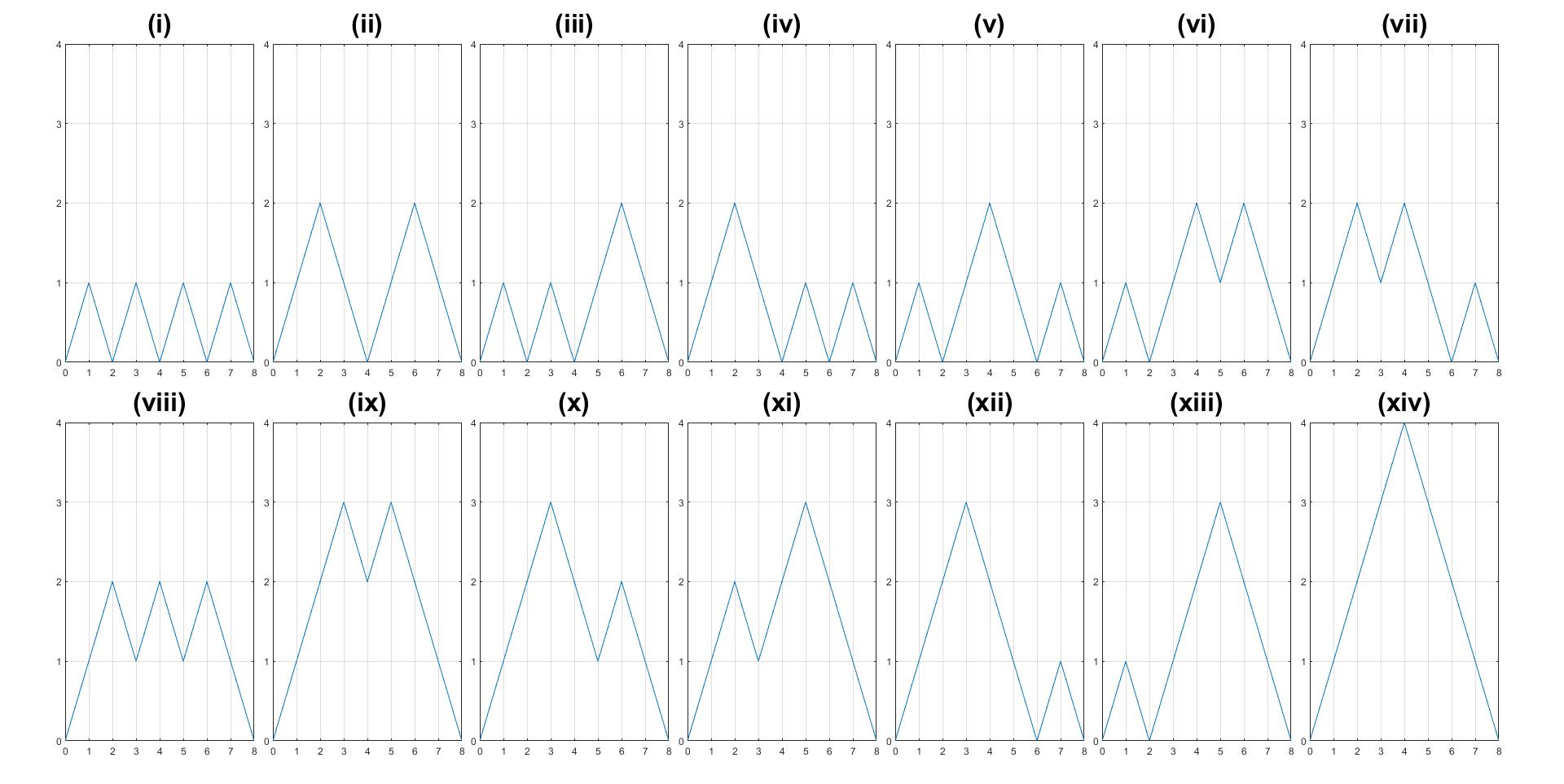}
\caption{List of Dyck paths of semilength 4}
\label{Dyckpath4}
    \end{center}
\end{figure}

The relationship between Symmetric Dyck paths and  Chebyshev polynomials has been discussed in \cite{Cigler,GuibertMansour}. Theorem \ref{th:Dyck2} presents an asymptotic expansion of $g_d$ expressed as a power series, followed by a proof.

\begin{theorem}\label{th:Dyck2}

We consider the family of functions $g_d(x)=\left\{    \begin{array}{lllllllllll} \frac{2U_{\frac{d}{2}-1}(x)}{V_{\frac{d}{2}}(x)} & \text{if   $d$ is even } \\     \frac{W_{\frac{d-1}{2}}(x)}{T_{\frac{d+1}{2}}(x)} & \text{if   $d$ is odd }\end{array}\right.$.  Let  $\Delta^S_{n,h}$ be the number of Symmetric Dyck paths of semilength $n$ and of maximum height $h$. Over the interval $[1,\infty[$, $g_d$ can be expressed using power series as follows:

\begin{equation}
g_d(x)=\sum_{i=0}^\infty \frac{\Delta^S_{i,d-1}}{2^{i}x^{i+1}}
\label{gdth}
\end{equation}

\end{theorem}

\begin{proof}

A similar approach using involutions, generating functions and Equation (\ref{CHEGU1}) has been presented in \cite{GuibertMansour}. Let us denote $\displaystyle \widehat{g_d}(x)=\sum_{i=0}^\infty \frac{\Delta^S_{i,d-1}}{2^{i}x^{i+1}}$. We will prove by induction that $\widehat{g_d}=g_d$. The case $d=1$ is immediate. Let us assume $\widehat{g_d}=g_d$ for $d$ greater or equal than 1 and we develop Equation (\ref{CHEGU2}) using the Cauchy product:

\begin{equation}
\begin{array}{llllll}
g_{d+1}&=& x f_{d+1}(x)(2+g_{d}(x)) \\
&=&\displaystyle x\left(\sum_{i=0}^\infty \frac{\Delta_{i,d}}{2^{2i+1}x^{2i+2}}\right)\left(2+\sum_{i=0}^\infty \frac{\Delta^S_{i,d-1}}{2^{i}x^{i+1}}\right) \\
&=&\displaystyle x\left(\sum_{i=0}^\infty \frac{\Delta_{i,d}}{2^{2i+1}x^{2i+2}}\right)\left(\sum_{i=-1}^\infty \frac{\Delta^S_{i,d-1}}{2^{i}x^{i+1}}\right)
= \displaystyle  \frac{1}{x} \left(\sum_{i=0}^\infty \frac{\Delta_{i,d}}{2^{2i}x^{2i}}\right)\left(\sum_{i=0}^\infty \frac{\Delta^S_{i-1,d-1}}{2^{i}x^{i}}\right)\\
&=& \displaystyle  \frac{1}{x} \left(\sum_{i=0}^\infty \sum_{k=0}^{\myfloor{\frac{i}{2}}} \frac{\Delta_{k,d}}{2^{2k}} \times \frac{\Delta^S_{i-1-2k,d-1}}{2^{i-2k}} x^{-i} \right) \\
&=& \displaystyle \sum_{i=0}^\infty \frac{\Delta^S_{i,d}}{2^{i}x^{i+1}} \\
&=& \widehat{g_{d+1}}
\end{array}
\label{cauchy2deltaS}
\end{equation}

\end{proof}

The sequence $\left\{\Delta^S_{n,h}\right\}_{\{n,h \geq 1\}}$ corresponds to the OEIS sequence {A94718}. Closed form expressions are presented in \cite{Cigler,Broadway}.

\end{document}